\documentclass{amsart}
\usepackage{amsfonts,amssymb,amsmath,amsthm}
\usepackage{url}
\usepackage{enumerate}
\usepackage{mathrsfs}
\usepackage{fullpage}
\numberwithin{equation}{section}
\def\a{\alpha}
\def\b{\beta}

\def\gm{\gamma}

\def\s{\sigma}

\def\C{\ensuremath{\mathbb{C}}}

\def\N{\ensuremath{\mathbb{N}}}
\def\R{\ensuremath{\mathbb{R}}}

\def\Rd{ {\R^d}}

\def\pd{\partial}

\def\qqq{\quad\forall\,}

\newcommand{\cP}{\ensuremath{\mathcal{P}}}

\newcommand{\abs}[1]{\left\|{#1}\right\|}
\newcommand{\absv}[1]{\left|{#1}\right|}
\newcommand{\norme}[1]{\left\|{#1}\right\|}

\newcommand{\scal}[1]{\left\langle\relax #1 \relax\right\rangle}

\newcommand{\qtxtq}[1]{\quad \text{#1}\quad}

\newtheorem{thm}{Theorem}[section]
\newtheorem{prop}[thm]{Proposition}
\newtheorem{lem}[thm]{Lemma}
\newtheorem{cor}[thm]{Corollary}

\theoremstyle{definition}

\theoremstyle{remark}
\newtheorem*{rmk}{Remark}

\author{L. Deleaval}
\address{Laboratoire d'Analyse et de Math\'ematiques appliqu\'ees \\ Universit\'e Paris-Est Marne-la-Vall\'ee \\
France}
\email{luc.deleaval@u-pem.fr}
\author{N. Demni}
\address{Institut de Recherche en Math\'ematiques de Rennes\\ Universit\'e Rennes 1\\
France}
\email{nizar.demni@univ-rennes1.fr}
\author{H. Youssfi}
\address{Centre de Math\'ematiques et d'informatique\\ Universit\'e Aix-Marseille I \\ France} 
\email{Hassan.Youssfi@cmi.univ-mrs.fr}
 
\keywords{Dunkl operators; Dunkl kernel; Dihedral root systems; Generalized Bessel function}
\subjclass[2010]{33C52; 33C80; 43A80}

\begin{document}

\title{Dunkl kernel associated with dihedral groups
}

\thanks{The authors would like to thank Professor Charles Dunkl for valuable comments}

\begin{abstract}
In this paper, we pursue the investigations started in \cite{Mas-You} where the authors provide a construction of the Dunkl intertwining operator for a large subset of the set of regular multiplicity values. More precisely, we make concrete the action of this operator on homogeneous polynomials when the root system is of dihedral type and under a mild assumption on the multiplicity function. In particular, we obtain a formula for the corresponding Dunkl kernel and another representation of the generalized Bessel function already derived in \cite{Demni0}. When the multiplicity function is everywhere constant, our computations give a solution to the problem of counting the number of factorizations of an element from a dihedral group into a fixed number of (non necessarily simple) reflections. In the remainder of the paper, we supply another method to derive the Dunkl kernel associated with dihedral systems from the corresponding generalized Bessel function. This time, we use the shift principle together with multiple combinations of Dunkl operators in the directions of the vectors of the canonical basis of $\mathbb{R}^2$. When the dihedral system is of order six and only in this case, a single combination suffices to get the Dunkl kernel and agrees up to an isomorphism with the formula recently obtained by Amri \cite[Lemma1]{Amri} in the case of a root system of type $A_2$. We finally derive an integral representation for the Dunkl kernel associated with the dihedral system of order eight.
\end{abstract}
\maketitle

\section{Reminder and motivation}
In his seminal paper \cite{Dunkl}, C.F. Dunkl introduced a deformation of the usual partial derivatives by reflections, the so-called differential-difference operators which are now commonly named Dunkl operators.  They form a commutative algebra which generalizes the algebra of invariant differential operators on Euclidean symmetric spaces. The study of Dunkl operators leads to a rich harmonic analysis which extends the Euclidean Fourier analysis to arbitrary reduced root systems in a finite-dimensional vector space and multiplicity functions (see \cite[chapters 4 and 5]{Dunkl-Xu} for a detailed account). In particular, an analogue of the exponential function, referred to as the Dunkl kernel, is defined as the unique smooth common eigenfunction of Dunkl operators. Equivalently, the latter is the image of the former under the action of the so-called Dunkl intertwining operator. As a matter of fact, explicit expressions for this kernel or equivalently for the action of the intertwining operator are of great relevance for developing the harmonic analysis of Dunkl operators. 
However, obtaining them remains up to now a challenging problem and they are only known for few particular cases. For instance, when the root system is of type $B_1$, the Dunkl kernel is a combination of the modified Bessel function of the first kind and of its first derivative. For the rank-two root systems of types $A_2, B_2$, multiple integral representations were derived in \cite{Amri}, \cite{Dunkl1} and \cite{Dunkl2}: the key tool in the first of these papers is the so-called shift principle (\cite[Proposition 1.4]{Dunkl1}) while the last ones rely heavily on Harish-Chandra integral representations for the unitary and the symplectic groups respectively. In \cite{Mas-You}, a multiple integral representation of the Dunkl intertwining operator associated with an arbitrary orthogonal root system was proved and subsequently exploited in \cite{Del} in order to get the corresponding generalized translation operator. When the root system is of dihedral-type, the action of this operator on the monomial basis was described in \cite{Dunkl3}. More generally, a construction of the intertwining operator corresponding to an arbitrary root system with positive multiplicity values relies on exponential of matrices in the reflection group algebra as well as a variant of Poincar\'e lemma for Dunkl operators (see \cite[p.160-162]{Dunkl-Xu}). Another approach to this construction was the main object of \cite{Mas-You} where the action on the space of homogeneous polynomials of a fixed degree was described by means of the resolvent of an element from the reflection-group algebra leaving invariant this space. In this description, the existence of the resolvent is restricted to a proper, yet large, subset of the set of regular multiplicity values including those with nonnegative real parts (\cite{Mas-You1}).

The aim of this paper is two fold. Firstly, we pursue the investigations started in \cite{Mas-You}: under a mild assumption on the multiplicity function, the resolvent alluded to above is expanded as a convergent operator-valued (acting on homogeneous polynomials) power series. Viewed as an element in the reflection group algebra and for nonnegative multiplicity values, the coefficients of this expansion are identified by the virtue of \cite[Proposition 4.5.8, (ii)]{Dunkl-Xu} with those used in Dunkl's original construction of the intertwining operator (see \cite[Definition 4.5.7]{Dunkl-Xu}). When the multiplicity function takes a single complex value, they are proportional to the number of all factorizations of group elements into products of (non necessarily simple) reflections. In particular, when the root system is of type $A$, they coincide up to a multiplicative factor with the connection coefficients for the symmetric group relative to the orbit formed by the set of all transpositions (\cite{Gou-Jac}). For dihedral root systems, we compute them using the fact that dihedral groups contain only reflections and rotations and write a bijective proof. For even dihedral systems with arbitrary multiplicity values, we follow a different method in order to compute the sought coefficients and retrieve those already computed when the multiplicity function takes a single value. Doing so leads to concrete formulas for the action of the intertwining operator on homogeneous polynomials and in turn for the Dunkl kernel associated with dihedral systems. Besides, averaging the latter over a dihedral group leads to another representation of the generalized Bessel function already derived in \cite{Demni0} relying on probabilistic techniques. 

Secondly, we supply another method to derive the Dunkl kernel associated with dihedral systems from the corresponding generalized Bessel function. This fashion may be seen as a generalization to all dihedral systems of \cite[Lemma1]{Amri} valid for the root system of type $A_2$. However, apart from the shift principle, we need to apply multiple combinations of Dunkl operators in the directions of the vectors of the canonical basis of $\mathbb{R}^2$. When the dihedral system is of order six, a single combination suffices to get the corresponding Dunkl kernel and agrees up to an isomorphism with \cite[Lemma1]{Amri}. Finally, when the multiplicity function is non negative, we derive an integral representation for the Dunkl kernel associated with the dihedral system of order eight. 

The paper is organized as follows. In order to make the exposition self-contained, we recall in the next section the main features of the construction of the Dunkl intertwining operator given in \cite{Mas-You}. Section 3 is devoted to our further investigations related to this construction especially for dihedral systems. In the last section, we write down the multiple combinations of Dunkl operators needed to derive the Dunkl kernel from the corresponding generalized Bessel function. We also illustrate there our computations for dihedral systems of orders six and eight and derive the aforementioned integral representation for the Dunkl kernel associated with the latter.  
 
\section{Root systems, Dunkl operators and the intertwining operator} 
For facts on root systems, we refer the reader to the monograph \cite{Hum}. Let $(V,\langle \cdot, \cdot \rangle)$ be a finite dimensional Euclidean space and denote by $\mathopen\|\cdot\mathclose\|:=\langle \cdot, \cdot \rangle^{1/2}$ the corresponding Euclidean norm. A root system $R$ in $V$ is a finite set of vectors (called roots) in $V \setminus \{0\}$ such that 
\begin{equation*}
\forall \alpha \in R, \quad \sigma_{\alpha}(R) = R, 
\end{equation*}
where $\sigma_{\alpha}$ is the reflection with respect to $\alpha^{\perp}$: 
\begin{equation*}
\forall x \in V, \quad \sigma_{\alpha}(x) = x - 2\frac{\langle \alpha, x\rangle}{\langle \alpha, \alpha\rangle}\alpha. 
\end{equation*}
In order to introduce the Dunkl operators, we assume that the root system is reduced, that is
\begin{equation*}
\forall \alpha \in R, \quad R \cap \mathbb{R} \alpha = \{\pm \alpha\}
\end{equation*}
but not necessarily crystallographic. The set of reflections generates a finite group $G$, called the reflection group associated with $R$. It acts on functions as    
\begin{equation*}
(g\cdot f)(x):= f(gx),\quad g \in G, x\in V.
\end{equation*}

With this action in mind, a function  $k:R  \to  \C$ is a multiplicty function if it is a $G$-invariant function
\begin{equation*}
\forall g \in G, \forall \alpha \in R, \quad k(g\alpha)=k(\alpha). 
\end{equation*}
Therefore, it takes as many values as the number of orbits of $G$ acting on $R$. 
Since $\{\alpha^\perp, \alpha \in R\}$ is a finite set of hyperplanes, we can choose $\b\in V$ such that $\scal{\a,\b}\ne 0$ for all $\a\in R$. Doing so gives rise to a partial order in $R$ and the set $R_+:=\{\a\in R:\,\scal{\a,\b}>0\}$ is called a positive system. 

Given $(R, R_+, k)$ and $\xi\in V$, the corresponding Dunkl operator $T_\xi:=T_\xi(k)$ is defined for smooth functions $f$ by 
\begin{equation}\label{eq: Dunkl Op}
T_\xi f(x)=\pd_\xi f(x)+\sum_{\a\in R_+}k(\a)\scal{\a,\xi}\frac{f(x)-f(\s_{\a}x)}{\scal{\a,x}}, \quad x \in V,
\end{equation}
where $\pd_\xi$ is the usual directional derivative.
Since $\sigma_\alpha(\alpha)=-\alpha$, then $R = R_+ \cup (-R_+)$ is a disjoint union and the $G$-invariance of $k$ therefore implies that $T_\xi$ does not depend on the choice of $R_+$. 

Now, for every $n \geq 0$, denote $\cP_n$  the space of homogeneous polynomials on $V$ of degree $n$ and
\begin{equation*}
M^{reg}:=\Bigl\{k: \,\bigcap_{\xi \in V}\ker\bigl(T_\xi(k)\bigr)=\C\cdot 1\Bigr\}
\end{equation*}
the set of regular multiplicity functions. Then, it has been shown in \cite{Dunkl0} and \cite{DJO} that for each $k\in M^{reg}$, there exists a unique isomorphism $V_k$ of $\mathcal P := \bigoplus_{n\geq0}\mathcal P_n$ which satisfies the following properties
\begin{equation*}
V_k(\cP_n)\subset \cP_n, \quad  V_k(1)=1\qtxtq{and}  T_\xi V_k=V_k\pd_\xi, \quad \xi \in V.
\end{equation*}
Thus $V_k$ intertwines the algebras of Dunkl operators and partial derivatives and for that reason is known as the Dunkl intertwining operator.
In \cite[Theorem A]{Mas-You} (see also \cite{Mas-You1}), a construction of $V_k$ was given when $k$ belongs to a proper subset of $M^{reg}$ and is as follows. On the group algebra of $G$, consider the element $A$ given by
\begin{equation}\label{Elem}
A := \sum_{\alpha\in R_+}k(\alpha)\, \s_\alpha
\end{equation}
which leaves invariant the space $\cP_n, n \geq 0$. Hence, $A_n:=A_{|\cP_n}$ is an endomorphism of $\cP_n$. Let $M^*$ be the set of multiplicity functions for which the operator
\begin{equation*}
(n+\gamma)- A_n, \quad \gamma := \sum_{\alpha \in R_+} k(\alpha),
\end{equation*}
is invertible for all $n \geq 1$ with inverse (the resolvent of $A_n$ at $n+\gamma$)
\begin{equation*}
H_n:=\bigl((n+\gamma)- A_n\bigr)^{-1}, \quad n \geq 1.
\end{equation*}
Then, it was proved in \cite{Mas-You} and \cite{Mas-You1} that $M^* \subsetneq M^{reg}$ and that the intertwining operator $V_k$ acts on any polynomial $p \in \cP_n$ as 
\begin{equation}\label{Action}
V_k(p)(x) = (\partial_x H)^n(p), \quad x \in V,
\end{equation}
where $H$ is the operator acting on polynomials on $V$ whose restriction to $\cP_n$ is $H_n$ and 
\begin{equation*}
(\pd_x H)^0:= \mathrm{Id}, \quad  (\pd_x H)^m=(\pd_x H)\circ(\pd_x H)^{m-1}, \, m \geq 1.
\end{equation*} 
Endowing $\cP_n$ with the supremum norm
\begin{equation*}
\norme{p}_\infty:=\sup_{\abs{x}\le 1}\absv{p(x)},
 \end{equation*}
it follows that the operator norm  of $A_n$ satifies $\|A_n\|_{\cP_n\to\cP_n} \leq \delta$, where 
\begin{equation*}
\delta := \sum_{\a\in R_+}\absv{k(\a)}.
\end{equation*}
As a matter of fact, if $\delta < |1+\gamma|$, then $k \in M^*$ and 
\begin{equation}\label{Exp}
H_n = \sum_{m=0}^{+\infty} \frac{A_n^m}{(n+\gm)^{m+1}}, \quad n \geq 1,
\end{equation}
where the series converges absolutely in the operator norm $\mathopen\|\cdot\mathclose\|_{\cP_n\to\cP_n}$. Clearly, the assumption $\delta < |1+\gamma|$ holds for nonnegative multiplicity functions and implies in general that $\Re(\gamma) > -1/2$ with equivalence when $k$ is everywhere constant. This elementary observation is the starting point of our subsequent investigations. 

\section{Dunkl kernel and generalized Bessel function}
Let $k \in M^*$ be such that $\delta < |1+\gamma|$. In this section, we first derive an elaborated version of \eqref{Action} for arbitrary root systems and focus afterwards on dihedral systems for which we obtain explicit expressions of the Dunkl kernel and the generalized Bessel function. To proceed, we introduce the following two sequences: for any $g \in G$, set 
\begin{equation*}
c_m(g) : =\sum_{\substack{(\a_{i_1},\ldots,\a_{i_m})\in R_+^m \\ \s_{\a_{i_1}}\cdots\s_{\a_{i_m}}=g}} k(\a_{i_1})\ldots k(\a_{i_m}), \quad m \geq 1, \quad c_0(g) = \delta_{ge}, 
\end{equation*}
where $e$ is the identity element of $G$ and
\begin{equation*}
C_n(g):= \sum_{m=0}^{+\infty}\frac{c_m(g)}{(n+\gm)^{m+1}}.
\end{equation*}
Note that the series defining $C_n(g)$ converges absolutely for all positive integer $n$ since $G$ is finite and
\begin{equation*}
\sum_{g \in G} |c_m(g)| \leq  \delta^m. 
\end{equation*}
The first result of this section is the following proposition.
\begin{prop}\label{Action1} 
For any $n \geq 1$, the action of $H$ on $\cP_n$ is given by 
\begin{equation*}
H_{|\cP_n} = H_n = \sum_{g\in G}C_n(g)g,
\end{equation*}
and in turn 
\begin{equation*}
V_k(p)(x)=  \sum_{g_1,\ldots,g_n\in G}C(g_1,\ldots,g_n)\pd_{g_1 x}\pd_{g_2x}\ldots\pd_{g_n x}p, \quad p \in \cP_n, \, x \in V,
\end{equation*}
where we set for every $g_1\in G, \ldots, g_n \in G$,
\begin{equation*}
C(g_1,\ldots,g_n):=C_n(g_n)C_{n-1}(g_n^{-1}g_{n-1})\ldots C_{1}(g_2^{-1}g_1).
\end{equation*}
\end{prop}

\begin{proof} 
Let $m \geq 1$ and use the definition \eqref{Elem} in order to compute: 
\begin{align*}
A^m & =\sum_{\a_1,\ldots,\a_m\in R_+}k(\a_1)\ldots k(\a_m)\s_{\a_{i_1}}\cdots\s_{\a_{i_m}} 
\\&=\sum_{g\in G}\Biggl(\sum_{\substack{(\a_{i_1},\ldots,\a_{i_m})\in R_+^m \\ \s_{\a_{i_1}}\cdots\s_{\a_{i_m}}=g}} k(\a_{i_1})\ldots k(\a_{i_m})\Biggr)g
= \sum_{g\in G}c_m(g) g.
\end{align*}
Consequently, the action of $H_n$ readily follows from the expansion \eqref{Exp}. As to that of $V_k$, an induction on $n$ shows that for any $p \in \cP_n$
\begin{equation*}
(\pd_x H)^{n} p=\sum_{g_1,\dots,g_n\in G} C_1(g_1)\ldots C_n(g_n)\pd_{g_n\ldots g_1x}\pd_{g_n\ldots g_2 x}\ldots\pd_{g_n x}p.
\end{equation*} 
Performing the variable change 
\begin{equation*}
(g_n, g_ng_{n-1},\ldots, g_n\ldots g_1)\mapsto(g_n,\ldots,g_1),
 \end{equation*}
 we are done.
 \end{proof}
\begin{rmk}
From the very definition of $A_n$ and Proposition \ref{Action1}, we have for any $n \geq 1$
\begin{equation*}
\sum_{g \in G} [(n+\gamma) - A_n]C_n(g) g = e. 
\end{equation*}
When $k \geq 0$, this is exactly part (ii) of \cite[Proposition 4.5.8]{Dunkl-Xu}. 
\end{rmk}

Let $(x,y) \mapsto E_k(x,y)$ denote the Dunkl kernel corresponding to the root system $R$ and recall that (see \cite[section 4.6]{Dunkl-Xu} for instance)
\begin{equation*}
E_k(x,y) = \sum_{n=0}^{+\infty} E_n(x,y), \quad E_n(x,y):=\frac{1}{n!}V_k\bigl(\scal{\cdot,y}^n\bigr)(x).
\end{equation*}
For fixed $y \in V$, we apply the previous proposition to the homogeneous polynomial $\mathopen\langle \cdot,y\mathclose\rangle^n$ and, together with successive differentiations, we get the following expression.

\begin{cor}\label{Noyau}
Let  $(x, y) \in V^2$. Then, for every positive integer $n$, we have
\begin{equation*}
E_n(x,y) = \sum_{g_1,\ldots,g_n\in G}C(g_1,\ldots,g_n) \prod_{j=1}^n\scal{g_j x,y}.
\end{equation*}
\end{cor}
Besides, let 
\begin{equation*}
E_k^G(x,y) := \frac{1}{\# G}\sum_{g \in G}E_k(gx,y) =\frac{1}{\# G} \sum_{n = 0}^{+\infty}\sum_{g \in G}E_n(gx,y)
\end{equation*}
be the so-called generalized Bessel function. This is a $G$-invariant function and bears this name since it reduces to a modified Bessel function for the rank-one root system of type $B_1$. For further expressions of $E_k^G$ in higher ranks, we refer the reader to the last chapter of \cite{CDGRVY}. From Corollary \ref{Noyau}, we deduce the following expression for $E_k^G$.

\begin{cor}\label{Bessel} 
For any $(x, y) \in V^2$, 
\begin{equation*}
E^G_k(x,y) = 1+\frac{1}{\# G}\sum_{n = 1}^{+\infty}\frac{1}{n}\sum_{g_1,\ldots,g_n\in G}C_{n-1}(g_n^{-1}g_{n-1})\ldots C_{1}(g_2^{-1}g_{1})\prod_{j=1}^n\scal{g_j x,y}. 
\end{equation*}
\end{cor}

\begin{proof}
Using Corollary \ref{Noyau}, we compute
\begin{align*}
\sum_{g \in G}E_n(gx,y) &= \sum_{g \in G} \sum_{g_1,\ldots,g_n\in G}C(g_1,\ldots,g_n) \prod_{j=1}^n\scal{g_j gx,y} 
\\ &= \sum_{g_1,\ldots,g_n\in G}\left\{\sum_{g \in G} C(g_1g^{-1},\dots, g_ng^{-1})\right\} \prod_{j=1}^n\scal{g_j x,y}.
\end{align*}
Now, we claim that $c_m(wgw^{-1}) = c_m(g)$ for any $w, g \in G$ and any $m \geq 0$. Indeed, this fact is obvious when $m=0$ since $wgw^{-1} \neq e$ if $g \neq e$, while it follows when $m \geq 1$ from the fact that 
\begin{equation*}
w\s_{\a}w^{-1} = \s_{w\a}, \quad  \alpha \in R, 
\end{equation*}
together with $\s_{\a} = \s_{-\a}, \alpha \in R$. Hence, the same relation holds for $C_m(g)$ which in turn yields   
\begin{align*}
C(g_1g^{-1},\dots, g_ng^{-1}) &= C_n(g_ng^{-1})C_{n-1}(g_n^{-1}g_{n-1})\ldots C_{1}(g_2^{-1}g_{1}).
\end{align*}
Besides, it is clear that $c_m(g) = c_m(g^{-1})$ whence 
\begin{align*}
\sum_{g \in G}C_n(g_ng^{-1}) &=  \sum_{g \in G}C_n(g) = \sum_{m=0}^{+\infty}\frac{1}{(n+\gamma)^{m+1}}\sum_{g \in G}c_m(g).
\end{align*}
Finally,
\begin{align*}
\sum_{g \in G}c_m(g) = \sum_{\a_{i_1},\ldots,\a_{i_m}\in R_+^m} k(\a_{i_1})\ldots k(\a_{i_m}) = \gamma^m
\end{align*}
so that
\begin{align*}
\sum_{g \in G}C_n(g_ng^{-1}) =  \sum_{m=0}^{+\infty}\frac{\gamma^m}{(n+\gamma)^{m+1}} = \frac{1}{n}.
\end{align*}
\end{proof}

\begin{rmk}
In the last proof, we observed that $g \mapsto c_m(g)$ is a class function. In particular, when $G$ is the symmetric group, it takes constant values on partitions and $(c_m(g))_{m \geq 0, g \in G}$ are known as the connection coefficients relative to the orbit formed by the set of all transpositions (\cite{Gou-Jac}). On the other hand, the results proved below for dihedral groups show that $c_m$ assigns different values to the identity, the set of reflections and that of rotations. Apparently, $c_m(g)$ depends only on the co-dimension of the fixed subspace of $g$ which is nothing else but the length of $g$ with respect to the set of all reflections (\cite{Car}). 
\end{rmk}

We close this paragraph by the following important induction satisfied by the sequence $(c_m(g))_{m \geq 0}$:
\begin{lem}\label{lem: pties of c_n} 
For any integer $m \geq 0$ and any $g \in G$, 
\begin{equation*}
\sum_{\a\in R_+}k(\a) c_m(\s_{\a}g) = \sum_{\a\in R_+}k(\a) c_m(g\s_{\a})=c_{m+1}(g). 
\end{equation*}
\end{lem}
\begin{proof}
If $m=0$, then 
\begin{equation*}
\sum_{\a\in R_+}k(\a) c_0(g\s_{\a}) 
\end{equation*} 
vanishes unless $g$ is a reflection. But, the only reflections of $G$ are of the form $\sigma_{\alpha}, \alpha \in R$ (\cite{Hum}, p.24). Thus, the statement of the lemma follows in this case from the definition of $c_1(g)$. Otherwise, for any $m \geq 1$,
\begin{align*}
 \sum_{\a\in R_+}k(\a) c_m(\s_{\a}g)&=\sum_{\a\in R_+}k(\a)\Biggl(\sum_{{(\a_1,\ldots,\a_m)\in R_+^m}\atop\s_{\a_{1}}\cdots\s_{\a_{m}}=\s_{\a}g} k(\a_1)\ldots k(\a_n)\Biggr)
 \\&=\sum_{{(\a,\a_1,\ldots,\a_m)\in R_+^{m+1}}\atop\s_{\a}\s_{\a_{1}}\cdots\s_{\a_{m}}=g} k(\a) k(\a_1)\ldots k(\a_n)\\
&=c_{m+1}(g).
\end{align*}
Recalling $c_m(g)=c_m(g^{-1})$, the lemma is proved.
\end{proof}

\subsection{Application to dihedral systems}
Recall from \cite{Hum} (see also \cite{Dunkl-Xu}) that for any integer $s \geq 2$, the dihedral system $I_2(s)$ is the subset of $V = \mathbb{R}^2 \approx \mathbb{C}$ defined by 
\begin{equation*}
I_2(s) := \{\pm i e^{ij\pi/s}, \, 1\leq j \leq s\}. 
\end{equation*}
Wa can choose as a  positive subsystem of $R$ the set of vectors $\{-i e^{ij\pi/s}, \, 1\leq j \leq s\}$ and the dihedral group $G = D_2(s)$ consists of $s$ reflections $\sigma_j$ and $s$ rotations $r_j$ written respectively in complex notations as
\begin{equation*}
\sigma_j : x \mapsto \overline{x}e^{2ij\pi/s}, \quad r_j : x \mapsto xe^{2ij\pi/s}, \, 1 \leq j \leq s.
\end{equation*}
When $s$ is odd, the roots form a single orbit so that a multiplicity function takes a single value. Otherwise, if $s=2q$ is even, then there are two orbits so that a multiplicity function takes at most two values. For sake of simplicity, we shall first consider the case of a constant multiplicity function: $k(\alpha) = k$ for all $\alpha \in I_2(s)$ and write a bijective proof of the result stated in the following proposition. Note that in this case, $c_m(g)$ counts  (up to a constant multiplicative factor) the number of factorizations of $g$ into a product of $m$ reflections. the solution given below relies on the fact that dihedral groups contain only reflections and rotations.
\begin{prop}\label{PropCoeff}
Assume $k(\alpha) = k$ for all $\alpha \in R$. Then for every $m \geq 1$, 
\begin{equation*}
c_m(g) =  k^m |R_+|^{m-1} = \frac{\gamma^m}{|R_+|} 
\end{equation*} 
if $g$ is a reflection and $m$ is odd or $g$ is a rotation and $m$ is even. Otherwise $c_m(g) = 0$.
\end{prop}
\begin{proof}
If $g$ is a reflection, then $c_m(g) = 0$ when $m$ is even, otherwise if $g$ is a rotation, then $c_m(g) = 0$ when $m$ is odd. So, assume for instance that $g$ is a rotation and $m$ is even. Then to any choice of 
$\sigma_{\a_2}\cdots\sigma_{\a_m}, \a_2, \ldots, \a_m \in R_+$, there exists a unique reflection $\sigma_1$ such that 
\begin{equation*}
\sigma_{1}\sigma_{\beta_2}\cdots \sigma_{\beta_m} = g.
\end{equation*}
Indeed, $g\sigma_{\beta_m}\cdots \sigma_{\beta_2}$ is a reflection and  therefore must be of the form $\sigma_{\beta_1}$ for some $\beta_1 \in R$ since the reflections of any finite reflection group $G$ are only of this form (\cite[p.24]{Hum}). As a matter of fact, the number of possible ways of writing $g$ as a product of $m$ reflections is exactly $k^m$ times the total number of choices of $m-1$ elements of $R_+$. A similar reasoning applies when $g$ is a reflection and $m$ is odd.  
\end{proof}

Before dealing with the case of distinct multiplicity values which only occurs for even dihedral systems, we readily compute $C_n(g)$ when $k$ takes a single value.
\begin{cor}
Let $n \geq 1$. 
\begin{enumerate}
\item If $g = \mathrm{Id}$, then 
\begin{equation*}
C_n(\mathrm{Id}) = \frac{(n+\gamma)}{n|R_+|(n+2\gamma)}.
\end{equation*}
\item If $g \neq \mathrm{Id}$ is a rotation, then 
\begin{equation*}
C_n(g) = C_n(\mathrm{Id}) - \frac{1}{|R_+|(n+\gamma)} = \frac{\gamma^2}{n|R_+|(n+\gamma)(n+2\gamma)}.
\end{equation*}
\item If $g$ is a reflection, then 
\begin{equation*}
C_n(g) = \frac{\gamma}{(n+\gamma)}C_n(\mathrm{Id}) = \frac{\gamma}{n|R_+|(n+2\gamma)}.
\end{equation*}
\end{enumerate}
\end{cor} 

Now, let $s=2q, q \geq 2$, and assign the values $k_1 \neq k_2$ to the orbits corresponding respectively to the sets of reflections 
\begin{equation*}
\mathcal{O}_0 := \{\sigma_{2j}, 1 \leq j \leq q\}, \quad \mathcal{O}_1 := \{\sigma_{2j+1}, 0 \leq j \leq q-1\}. 
\end{equation*}
For an element $g \in D_2(s)$, write $g^{(+)}$ or $g^{(-)}$ if $m$ is odd and $g$ belongs to $\mathcal{O}_0, \mathcal{O}_1$ respectively, or $m$ is even and $g$ belongs to 
\begin{equation*}
\mathcal{O}_2 := \{r_{2j}, 1 \leq j \leq q\}, \quad \mathcal{O}_3 := \{r_{2j+1}, 0 \leq j \leq q-1\}
\end{equation*}
respectively. We have the following proposition.
\begin{prop}
For any $m \geq 1$, 
\begin{equation*}
c_m(g^{(+)}) = \frac{q^{m-1}}{2} \left[(k_1+k_2)^m + (k_1-k_2)^m\right], \,\, c_m(g^{(-)}) = \frac{q^{m-1}}{2}\left[(k_1+k_2)^m - (k_1-k_2)^m\right].
\end{equation*}
\end{prop}
\begin{proof} 
Recall that 
\begin{equation}\label{Ind-Sch}
c_{m+1}(g) = \sum_{\alpha \in R_+}k(\alpha)c_m(\sigma_{\alpha}g), \quad c_0(g) = \delta_{ge}. 
\end{equation}
Then we readily get 
\begin{equation*}
c_1(\sigma_{2j}) = k_1, \quad c_1(\sigma_{2j+1}) = k_2. 
\end{equation*}
More generally, \eqref{Ind-Sch} shows that for any $m \geq 1$, there exist two homogeneous polynomials $P_m, Q_m$ in two variables such that 
\begin{equation*}
c_m(g^{(+)}) = q^{m-1}P_m(k_1,k_2), \quad c_m(g^{(-)}) = q^{m-1}Q_m(k_1,k_2).
\end{equation*}
Indeed, it suffices to split the sum in the right hand side of \eqref{Ind-Sch} over $\mathcal{O}_0, \mathcal{O}_1$ and to observe that 
\begin{equation*}
\sigma_{2j} g^{(+)} \in \mathcal{O}_0 \cup \mathcal{O}_2, \quad \sigma_{2j+1} g^{(+)} \in \mathcal{O}_1 \cup \mathcal{O}_3,
\end{equation*}
while 
\begin{equation*}
\sigma_{2j} g^{(-)} \in \mathcal{O}_1 \cup \mathcal{O}_3, \quad \sigma_{2j+1} g^{(-)} \in \mathcal{O}_0 \cup \mathcal{O}_2.
\end{equation*}
Actually, the polynomials $(P_m)_m, (Q_m)_m$ are defined inductively by 
 \begin{eqnarray*}
P_m(k_1,k_2) & = k_1P_{m-1}(k_1,k_2)+k_2Q_{m-1}(k_1,k_2), \quad m \geq 2, \\
Q_m(k_1,k_2) & = k_1Q_{m-1}(k_1,k_2)+k_2P_{m-1}(k_1,k_2), \quad m \geq 2,
\end{eqnarray*}
with the initial values $P_1(k_1,k_2) = k_1, Q_1(k_1,k_2) = k_2$. Equivalently, 
\begin{equation*}
\left(\begin{array}{c}
P_m(k_1,k_2) \\ Q_m(k_1,k_2) 
\end{array}\right) = \left(\begin{array}{lr}
k_1 & k_2 \\ 
k_2 & k_1 
\end{array}
\right) \,\, \left(\begin{array}{c}
P_{m-1}(k_1,k_2) \\ Q_{m-1}(k_1,k_2) 
\end{array}\right), \quad \left(\begin{array}{c}
P_1(k_1,k_2) \\ Q_1(k_1,k_2) 
\end{array}\right) = \left(\begin{array}{c}
k_1 \\ k_2 
\end{array}\right).
\end{equation*}
Consequently 
\begin{equation*}
\left(\begin{array}{c}
P_m(k_1,k_2) \\ Q_m(k_1,k_2) 
\end{array}\right) =  \left(\begin{array}{lr}
k_1 & k_2 \\ 
k_2 & k_1 
\end{array}
\right)^{m-1} \, \left(\begin{array}{c}
k_1 \\ k_2 
\end{array}\right).
\end{equation*}
When $k \in M^*$ takes real values, then the matrix displayed in the right hand side is real symmetric with eigenvalues $k_1\pm k_2$ and one-dimensional eigenspaces spanned by the vectors 
\begin{equation*}
(1,1) , \quad (1,-1).
\end{equation*}
Hence, its eigenvalues decomposition leads to 
\begin{align*}
\left(\begin{array}{c}
P_m(k_1,k_2) \\ Q_m(k_1,k_2) 
\end{array}\right) &= \left(\begin{array}{lr}
1& 1 \\ 
1 & -1 
\end{array}
\right)\left(\begin{array}{cc}
(k_1+k_2)^{m-1} & 0 \\ 
0 & (k_1-k_2)^{m-1} 
\end{array}
\right)\left(\begin{array}{c}
k_1+k_2 \\ k_1-k_2
\end{array}\right)
\\& = \frac{1}{2}\left(\begin{array}{c}
(k_1+k_2)^m + (k_1-k_2)^m \\ (k_1+k_2)^m - (k_1-k_2)^m
\end{array}\right).
\end{align*}
For complex values of $k \in M^*$, the result is proved by induction on $m$. 
\end{proof}

With these findings, we recover the result of Proposition \ref{PropCoeff} by taking $k_1 = k_2$ and get explicit expressions of $E_k$ and $E_k^G$ for arbitrary multiplicity values.  

\section{Dunkl kernel of dihedral-type: another approach} 
In this section, we supply another method to derive the Dunkl kernel associated with dihedral systems from the corresponding generalized Bessel function (\cite{Demni0}). Though this fashion allowed recently to get a multiple integral representation for $E_k$ when the root system is of type $A_2$ (\cite{Amri}) which is isomorphic to $I_2(3)$, its adaptation to other dihedral systems is not straightforward. Indeed, the shift principle  leads to the average of $E_k$ over the set of rotations rather than the whole group and we easily see that a first-order differential-difference isolates the term corresponding to the identity element only when the root system is of type $A_2$. For general dihedral root systems, we need to apply multiple combinations of the Dunkl operators in the directions of the vectors of the canonical basis of $V = \mathbb{R}^2$.

For sake of simplicity, we write the proof only for even dihedral systems and illustrate afterwards our findings for $I_2(4)$ and $I_2(3)$. In order to state our result, let us recall the following instance of the shift principle (see \cite{Dunkl}): if $k+1$ is the shift of the multiplicity function $k$ by $+1$ over all the roots then 
\begin{equation*}
E_{k+1}^G(x,y) = \frac{\eta_k}{(|G|)h(x)h(y)} \sum_{g \in G}\det(g)E_k(x,gy), \quad x\in V, y \in V. 
\end{equation*}
Here, $\det(g)$ is the determinant of $g$, $h$ is the fundamental alternating polynomial given for every $x \in V$ by
\begin{equation*}
h(x) = \prod_{\alpha \in R_+}\langle \alpha, x \rangle
\end{equation*}
and (\cite[Definition 9.4, p.368]{Opd}): 
\begin{equation*}
\eta_k = h(T)[h] = \prod_{\alpha\in R_+}T_{\alpha}h.  
\end{equation*}
We also make use of the notations\footnote{We omit the dependence on $k$ for sake of simplicity.} 
\begin{equation*}
U(x,y) := \frac{(|G|)}{2}\left\{E_k^G(x,y) + \frac{1}{\eta_k}h(x)h(y)E_{k+1}^G(x,y)\right\}
\end{equation*}
and
\begin{equation*}
T := \frac{1}{2}(T_1 - iT_2), \quad \overline{T} :=  \frac{1}{2}(T_1 + iT_2),
\end{equation*}
where $ T_1:= T_{e_1}, T_2:= T_{e_2}$.
\begin{prop}
Suppose that $s=2q, q \geq 2$, is an even integer and let $\omega_q := e^{i\pi/q}.$ Then the Dunkl kernel associated with $I_2(s)$ reads
\begin{equation*} 
2y\prod_{j=1}^{q-1} \left[i\Im\left(\omega_q^j\overline{y}\right)\right] E_k(x,y)=  \left[y + 2\overline{T}\right]\prod_{j=1}^{q-1}\left[ \omega_q^{j}T - \overline{\omega_q^{j}}\overline{T}\right]U(\cdot, y)(x),
\end{equation*}
where $x,y \in \mathbb{R}^2 \approx \mathbb{C}$. 
\end{prop}

\begin{proof}
Since $E_k(x,0_V) = 1$ for any $x \in \mathbb{R}^2$, we assume without loss of generality that $y \neq 0_V$. Now, the dihedral group $D_2(s)$ consists only of reflections and rotations so that the shift principle yields 
\begin{equation*}
U(x,y) = \sum_{j=0}^{s-1}E_k(x,r_jy) = \sum_{j=0}^{s-1}E_k(x,e^{ij\pi/q}y). 
\end{equation*}
Since the rotations $r_j: y \mapsto e^{ij\pi/q}y$ come by pairs corresponding to indices $\{j, j+q\}, 0 \leq j \leq q-1$ then 
\begin{equation*}
U(x,y) = \sum_{j=0}^{q-1}\left[E_k(x,e^{ij\pi/q}y) + E_k(x,-e^{ij\pi/q}y)\right]. 
\end{equation*}
Next, the actions $T_iE_k(\cdot,y)(x) = y_iE_k(x,y), i=1,2$ (\cite{Dunkl0}) entail
\begin{eqnarray*}
T_1U(\cdot,y)(x) &=&   \sum_{j=0}^{q-1}\langle e^{-ij\pi/q}, y \rangle \Bigl(E_k(x, e^{ij\pi/q}y)- E_k(x, -e^{ij\pi/q}y)\Bigr)\\ 
T_2 U(\cdot,y)(x) &=& \sum_{j=0}^{q-1}\langle ie^{-ij\pi/q}, y \rangle \Bigl(E_k(x, e^{ij\pi/q}y)- E_k(x, -e^{ij\pi/q}y)\Bigr).
\end{eqnarray*}
Therefore we get
\[
\left[\langle ie^{-i(q-1)\pi/q}, y \rangle T_1 - \langle e^{-i(q-1)\pi/q}, y \rangle T_2\right]U(\cdot,y)(x) = \sum_{j=0}^{q-2}b_{j}(y)\Bigl(E_k(x, e^{ij\pi/q}y)- E_k(x, -e^{ij\pi/q}y)\Bigr)
\]
with
\begin{equation*}
b_{j}(y) := \langle e^{-ij\pi/q}, y\rangle\langle ie^{-i(q-1)\pi/q}, y\rangle -  \langle e^{-i(q-1)\pi/q}, y\rangle \langle ie^{-ij\pi/q}, y\rangle.
\end{equation*}
Iterating this procedure again $(q-2)$ times leads to  
\begin{equation*}
\prod_{j=1}^{q-1}\left[\langle ie^{-ij\pi/q}, y\rangle T_1 - \langle e^{-ij\pi/q}, y\rangle T_2\right]U(\cdot, y)(x) = a(q,y)\Bigl(E_k(x, y) + (-1)^{q-1}E_k(x, -y)\Bigr).
\end{equation*}
Using complex notations, we have
\begin{eqnarray*}
2TU(\cdot,y)(x) &=&   \sum_{j=0}^{q-1} \overline{\omega_q^j y}  \Bigl(E_{k}(x, \omega_q^jy) - E_{k}(x, -\omega_q^jy)\Bigr)\\ 
2\overline{T} U(\cdot,y)(x) &=& \sum_{j=0}^{q-1} \omega_q^j y  \Bigl(E_{k}(x, \omega_q^jy) - E_{k}(x, -\omega_q^jy)\Bigr),
\end{eqnarray*}
so that
\begin{equation*}
\prod_{j=1}^{q-1}\left[ \omega_q^{j}T - \overline{\omega_q^{j}}\overline{T}\right] U(\cdot, y)(x) =   \prod_{j=1}^{q-1} \left[i\Im\left(\omega_q^{j}\overline{y}\right)\right] \Bigl(E_k(x, y) + (-1)^{q-1}E_k(x, -y)\Bigr).
\end{equation*}
Consequently, if $y \neq 0_{V}$ then 
\begin{equation*}
\left[y + 2\overline{T}\right]\prod_{j=1}^{q-1}\left[ \omega_q^{j}T - \overline{\omega_q^{j}}\overline{T}\right] U(\cdot, y)(x) = 2y\prod_{j=1}^{q-1} \left[i\Im\left(\omega_q^j\overline{y}\right)\right] E_k(x,y)
\end{equation*}
which completes the proof of the proposition.
\end{proof}

\subsection{Examples}
Let us apply the procedure described in the previous proof to the crystallographic dihedral systems $I_2(4)$ and $I_2(3)$. For the former, the rotations are given by 
\begin{equation*}
y \mapsto \pm y, \quad y \mapsto \pm (iy).
\end{equation*}
Hence
\begin{equation*}
\sum_{j=0}^3 E_k(x,r_jy) = \Bigl(E_k(x,y)+E_k(x,-y)\Bigr)+ \Bigl(E_k(x,iy)+E_k(x,-iy)\Bigr),
\end{equation*}
and as such
\begin{eqnarray*}
T_1U(\cdot,y)(x)&=& y_1\Bigl(E_k(x,y)-E_k(x,-y)\Bigl) -y_2\Bigl(E_k(x,iy)-E_k(x,-iy)\Bigl) \\ 
T_2U(\cdot,y)(x)&=& y_2\Bigl(E_k(x,y)-E_k(x,-y)\Bigl) +y_1\Bigl(E_k(x,iy)-E_k(x,-iy)\Bigl).
\end{eqnarray*}
It follows that  
\begin{equation}\label{Ident1}
\Biggl(y_1T_1+ y_2T_2\Biggr)U(\cdot,y)(x) =  (y_1^2+y_2^2)\left[E_k(x,y)-E_k(x,-y)\right]
 \end{equation}
whence
\begin{equation}\label{Ident2}
2\overline{T} \Biggl(y_1T_1+ y_2T_2\Biggr) U(\cdot,y)(x) = 2\overline{T} T_y U(\cdot,y)(x) = y(y_1^2+y_2^2)\left[E_k(x,y)+E_k(x,-y)\right].
\end{equation}
Combining \eqref{Ident1} and \eqref{Ident2}, we get 
\begin{equation}\label{Ident3}
\left[y+2\overline{T}\right]T_y U(\cdot,y)(x)  = 2y(y_1^2+y_2^2)E_k(x,y). 
\end{equation}
As to the latter $I_2(3)$, the shift principle yields 
\begin{equation*}
U(x,y) = E_k(x,y) + E_k(x,e^{2i\pi/3}y) + E_k(x,e^{4i\pi/3}y)
\end{equation*}
while the actions of the Dunkl operators $T_1, T_2$ entail
\begin{eqnarray*}
T_1U(\cdot, y)(x)  &=& y_1E_k(x,y)+ \langle e^{4i\pi/3}, y\rangle E_k(x,e^{2i\pi/3}y) + \langle e^{2i\pi/3}, y\rangle E_k(x,e^{4i\pi/3}y) \\
T_2U(\cdot, y)(x) &=& y_2E_k(x,y)+ \langle ie^{4i\pi/3}, y\rangle E_k(x,e^{2i\pi/3}y) + \langle ie^{2i\pi/3}, y\rangle E_k(x,e^{4i\pi/3}y). 
\end{eqnarray*}
Consequently, we get the equality
\begin{multline*}
\biggl(\langle ie^{2i\pi/3}, y\rangle T_1 - \langle e^{2i\pi/3}, y\rangle T_2\biggr)U(\cdot, y)(x) = \left[y_1\langle ie^{2i\pi/3}, y\rangle - y_2\langle e^{2i\pi/3}, y\rangle\right] E_k(x,y)+ 
\\ \left[\langle e^{4i\pi/3}, y\rangle \langle ie^{2i\pi/3}, y\rangle -   \langle ie^{4i\pi/3}, y\rangle \langle e^{2i\pi/3}, y\rangle\right] E_k(x,e^{2i\pi/3}y) 
\end{multline*}
to which we apply 
\begin{equation*}
\langle ie^{4i\pi/3}, y\rangle T_1 - \langle e^{4i\pi/3}, y\rangle T_2
\end{equation*}
in order to end with 
\begin{align*}
\left[y_1\langle ie^{4i\pi/3}, y\rangle - y_2\langle e^{4i\pi/3}, y\rangle\right] \left[y_1\langle ie^{2i\pi/3}, y\rangle - y_2\langle e^{2i\pi/3}, y\rangle\right] E_k(x,y).
\end{align*}
Note that only for this dihedral root system, we can find $(f_j(y))_{j=0}^2$ so that 
\begin{equation}\label{Op}
\left[f_0(y) + f_1(y)T_1 + f_2(y)T_2\right]U(\cdot, y)(x) =  E_k(x,y).
\end{equation}
Indeed, for any $y \in \mathbb{R}^2\setminus \{0\}$, this amounts to solve the system 
\begin{eqnarray}\label{SysEq}
f_0(y) + y_1 f_1(y) + y_2 f_2(y) & = & 1 \nonumber \\ 
f_0(y) + \langle e^{4i\pi/3}, y\rangle f_1(y) +  \langle ie^{4i\pi/3}, y\rangle  f_2(y) & = & 0 \\
f_0(y) + \langle e^{2i\pi/3}, y\rangle  f_1(y) +  \langle ie^{2i\pi/3}, y\rangle f_2(y) & = & 0 \nonumber .   
\end{eqnarray}
But since $1+e^{2i\pi/3} + e^{4i\pi/3} = 0$ then the sum of these equations yields $f_0(y) = 1/3$, while the last two equations show that $y_2f_1(y) = y_1f_2(y)$. Substituting this last relation in the first equation, we  get 
\begin{equation}\label{Coeffic}
f_2(y) = \frac{2}{3}\frac{y_2}{y_1^2+y_2^2}, \quad f_1(y) = \frac{2}{3}\frac{y_1}{y_1^2+y_2^2} 
\end{equation}
when $y_2 \neq 0$. Otherwise, if $y_2 = 0, y_1 \neq 0$ then the first equation determines uniquely $f_1(y)$ while the remaining ones show that $f_2(y) = 0$. Thus, \eqref{Coeffic} holds in both cases and we can see that the left hand side of \eqref{Op} agrees with the result announced in \cite[Lemma 1]{Amri}. More precisely, the root system of type $A_2$ is isomorphic to $I_2(3)$ via the map 
\begin{equation*}
X: \{(x_1,x_2,x_3) \in \mathbb{R}^3, x_1+x_2+x_3 = 0\} \mapsto \frac{1}{2} \left(\sqrt{3}(x_1+x_2), x_1-x_2\right)
\end{equation*}
whose inverse is given by 
\begin{equation*}
X^{-1}: (x_1,x_2) \mapsto \left(\frac{x_1}{\sqrt{3}} + x_2, \frac{x_1}{\sqrt{3}} - x_2, -2\frac{x_1}{\sqrt{3}}\right). 
\end{equation*}
Besides, if $A, B$ are the matrix representations of $X$ and $X^{-1}$ respectively then for any $\alpha$ in the root system of type $A_2$, 
\begin{equation*}
X\sigma_{\alpha}X^{-1} = \sigma_{X(\alpha)}, \quad B^TB = 2 {\bf I}_2,
\end{equation*}
where ${\bf I}_2$ is the $2 \times 2$ identity matrix. With regard to these relations, the Dunkl operators associated with the root systems of type $A_2$ and $I_2(3)$ are interrelated via the equivariance property 
\begin{equation*}
T_{\xi}(f)\circ X^{-1} = T_{X(\xi)}(f \circ X^{-1})
\end{equation*}
for any smooth function $f: \mathbb{R}^3 \rightarrow \mathbb{R}$ and any $\xi \in \mathbb{R}^3$. For odd dihedral systems $I_2(s), s \geq 4$, finding a single combination similar to the left hand side of \eqref{Op} leads to a system of 
$s$ equations with three indeterminates $(f_i(y))_{i = 0}^2$ as in \eqref{SysEq}. Summing all these equations readily yields that $f_0(y) = 1/s$ which shows that they are not compatible.  

\subsection{An integral representation of $E_k$ associated with $I_2(4)$}
In \cite{Amri}, an integral representation of $E_k$ was obtained for the rank-two root system $A_2$ and follows after tedious computations from the main result proved in \cite{Amr}. Appealing to the isomorphism between $A_2$ and $I_2(3)$, one obtains an integral representation of $E_k$ corresponding to the odd dihedral system $I_2(3)$. In this paragraph, we derive an integral representation of $E_k$ corresponding to $G = I_2(4)$, that is root system of type $B_2$, and positive values of $k$. To this end, we recall from \cite{Demni1} the following integral representation of $E_k^G$. Set $\nu := k_1+k_2$ where $k_1,k_2$ are the multiplicity values of the two orbits and denote
\begin{equation*}
\mathcal{I}_{\nu-1/2}(u) := \sum_{j \geq 0}\frac{1}{j!(\nu-1/2)_j}\left(\frac{u}{2}\right)^{2j}, \quad u \in \mathbb{R},
\end{equation*}
the modified Bessel function of index $\nu-1/2$ and normalized by $\mathcal{I}_{\nu-1/2}(0) = 1$ (\cite{AAR}). 
Then, for any
 \begin{equation*}
x=\rho e^{i\phi}, \, y = re^{i\theta}, \, \rho, r \geq 0, \, \phi, \theta \in [0,\pi/4]
\end{equation*}
lying in the positive Weyl chamber,
\begin{equation*}
E_k^G(x,y) = \int \mathcal{I}_{\nu-1/2}\left(\rho r\sqrt{\frac{1+u\cos(2\phi)\cos(2\theta) + v\sin(2\phi)\sin(2\theta)}{2}}\right)\mu^{k_1}(du)\mu^{k_2}(dv)
\end{equation*}
where 
\begin{equation*}
\mu^{k_j}(du) := \frac{\Gamma(k_j+1/2)}{\sqrt{\pi}\Gamma(k_j)}(1-u^2)^{k_j-1} {\bf 1}_{[-1,1]}(u)du, \quad j \in \{ 1,2\}.
\end{equation*}
In cartesian coordinates $x = (x_1, x_2), y= (y_1, y_2)$, we equivalently have 
\begin{equation*}
E_k^G(x,y) = \int \mathcal{I}_{\nu-1/2}\left(\sqrt{\frac{Z_{x,y}(u,v)}{2}}\right)\mu^{k_1}(du)\mu^{k_2}(dv)
\end{equation*}
where 
\begin{equation*}
Z_{x,y}(u,v) := (x_1^2+x_2^2)(y_1^2+y_2^2)+ u(x_1^2-x_2^2)(y_1^2-y_2^2)+ 4v(x_1x_2)(y_1y_2).
\end{equation*}
Now, the alternating polynomial reads 
\begin{equation*}
h(y) = \prod_{\alpha \in R_+}\langle \alpha, y \rangle = \frac{r^{4}}{8}\sin(4\theta) = \frac{1}{2}(y_1y_2)(y_1^2-y_2^2)
\end{equation*}
and the differentiation rule holds 
\begin{align*}
\partial_{uv} & \mathcal{I}_{\nu-1/2}\left(\sqrt{\frac{Z_{x,y}(u,v)}{2}}\right) = \frac{h(x)h(y)}{4\nu^2-1}\mathcal{I}_{\nu+3/2}\left(\sqrt{\frac{Z_{x,y}(u,v)}{2}}\right).
\end{align*}
Consequently 
\begin{equation*}
h(x)h(y)E_{k+1}^G(x,y) =  (4\nu^2-1) \\ \int \partial_{uv}\mathcal{I}_{\nu-1/2}\left(\sqrt{\frac{Z_{x,y}(u,v)}{2}}\right)\mu^{k_1+1}(du)\mu^{k_2+1}(dv)
\end{equation*}
which can be written after integration by parts as 
\begin{equation*}
h(x)h(y)E_{k+1}^G(x,y) =  4(4\nu^2-1)(k_1+1)(k_2+1) \int \mathcal{I}_{\nu-1/2}\left(Z_{x,y}(u,v)\right)(uv)\mu^{k_1}(du)\mu^{k_2}(dv).
 \end{equation*}
Furthermore, \cite[Theorem 4.11]{DJO} shows that
\begin{align*}
\eta_k &= 4\frac{(2k_1+1)(2k_2+1)}{(k_1+k_2+2)(k_1+k_2+1)}\prod_{j=1}^{s}(2(k_1+k_2)+j) 
\\& = 4\frac{(2k_1+1)(2k_2+1)}{(k_1+k_2+2)(k_1+k_2+1)}\frac{\Gamma(2(k_1+k_2+2) +1)}{\Gamma(2(k_1+k_2)+1)}.
\end{align*}
Thus, there exists a constant $\lambda_{k,q}$ depending only on $k,q$ such that 
\begin{equation*}
\frac{1}{4}U(x,y) = E_k^G(x,y) + \frac{h(x)h(y)}{\eta_k}E_{k+1}^G(x,y) =  \int \mathcal{I}_{\nu-1/2}\left(\sqrt{\frac{Z_{x,y}(u,v)}{2}}\right) \left(1+ \lambda_{k,q}(uv)\right)\mu^{k_1}(du)\mu^{k_2}(dv).
\end{equation*}
According to \eqref{Ident3}, we need to compute
\begin{equation*}
\left(y+\overline{T}\right)T_y U(\cdot,y)(x).
\end{equation*}
Obviously,
\begin{equation*}
\partial_y\mathcal{I}_{\nu-1/2}\left(\sqrt{\frac{Z_{\cdot,y}(u,v)}{2}}\right)(x) = \frac{1}{4(2\nu-1)}\left[\partial_yZ_{\cdot,y}(u,v)(x)\right] \mathcal{I}_{\nu+1/2}\left(\sqrt{\frac{Z_{x,y}(u,v)}{2}}\right).
\end{equation*}
Besides, recall that the four reflections of $D_2(4)$ are $\sigma_{e_1}, \sigma_{e_2}, \sigma_{e_1\pm e_2}$  
and note that their actions on $U(\cdot, y)$ are translated to $Z_{x,y}$ either as $v \mapsto -v$ or as $u\mapsto -u$.  
Since the measures $\mu^{k_j}, j =1,2,$ are symmetric then the changes of variables $v \mapsto -v, u\mapsto -u$ show that the contribution of the difference part of $T_y$ is
\begin{equation*}
 2\lambda_{k,q}\left(\sum_{\alpha \in R_+} k(\alpha)\frac{\langle \alpha,y\rangle}{\langle \alpha,x\rangle}\right)
\int \mathcal{I}_{\nu-1/2}\left(\sqrt{\frac{Z_{x,y}(u,v)}{2}}\right) (uv)\mu^{k_1}(du)\mu^{k_2}(dv).
\end{equation*}
Altogether, 
\begin{multline*}
y(y_1^2+y_2^2)E_k(x,y)= 2(y + \overline{T})\int \left\{\frac{1}{4(2\nu-1)}\left[\partial_yZ_{\cdot,y}(u,v)(x)\right] \mathcal{I}_{\nu+1/2}\left(\sqrt{\frac{Z_{x,y}(u,v)}{2}}\right) \left(1+ \lambda_{k,q}(uv)\right)\right.
\\ \left.+ 2\lambda_{k,q}\left(\sum_{\alpha \in R_+} k(\alpha)\frac{\langle \alpha,y\rangle}{\langle \alpha,x\rangle}\right) \mathcal{I}_{\nu-1/2}\left(\sqrt{\frac{Z_{x,y}(u,v)}{2}}\right) (uv)\right\}\mu^{k_1}(du)\mu^{k_2}(dv).
\end{multline*}

\begin{rmk}
When $q \geq 3$, the generalized Bessel function $E_k^G$ associated with $I_2(2q)$ admits an integral representation of the form 
\begin{equation*}
E_k^G(x,y) = \int f_q\left(\rho r, u\cos(q\phi)\cos(q\theta) + v\sin(q\phi)\sin(q\theta)\right)\mu^{k_1}(du)\mu^{k_2}(dv)
\end{equation*}
for some function $f_q$. If we assume further that $k_0+k_1$ is an integer, then an operational formula for $f_q$ was derived in \cite{Demni1}. However, we do not dispose of a simple formula for it by means of a single special function as we do when $q=2$.  
\end{rmk}

\end{document}